\documentclass{amsart}
\usepackage[francais,english]{babel}
\usepackage[T1]{fontenc}
\usepackage{amssymb,amsmath,dsfont}
\usepackage{amsrefs}

\newtheorem{theorem}{Theorem}
\newtheorem{proposition}{Proposition}

\newtheorem{lemma}{Lemma}

\newcommand{\B}{\mathbb{B}}

\newcommand{\R}{\mathbb{R}}

\newcommand{\N}{\mathbb{N}}

\def \crit {2^\star(s)}

\title{Hardy-Sobolev Equations on Compact Riemannian Manifolds}
\author{Hassan Jaber}
\address{Universit\'e de Lorraine, Institut Elie Cartan de Lorraine, UMR 7502, Vand\oe uvre-l\`es-Nancy, F-54506, France.}
\email{Hassan.Jaber@univ-lorraine.fr}
\date{October, \ 29th 2013.}
\begin{document}
\maketitle
\begin{abstract}
Let $(M,g)$ be a compact Riemannian Manifold of dimension $n \geq 3$, $x_0 \in M$, and $s \in (0,2)$. We let $\crit: = \frac{2(n-s)}{n-2}$ be the critical Hardy-Sobolev exponent. 
We investigate the existence of positive distributional solutions $u\in C^0(M)$ to the critical equation
$$\Delta_g u+a(x) u=\frac{u^{\crit-1}}{d_g(x,x_0)^s}\; \; \hbox{ in} \ M$$
where $\Delta_g:=-\hbox{div}_g(\nabla)$ is the Laplace-Beltrami operator, and $d_g$ is the Riemannian distance on $(M,g)$. Via a minimization method in the spirit of Aubin, we prove existence in dimension $n\geq 4$ 
when the potential $a$ is sufficiently below the scalar curvature at $x_0$. In dimension $n=3$, 
we use a global argument and we prove existence when the mass of the linear operator $\Delta_g + a$
is positive at $x_0$. As a byproduct of our analysis, we compute the best first constant for the related Riemannian Hardy-Sobolev inequality.  
\end{abstract}\vskip0.2cm

Let $(M,g)$ be a compact Riemannian Manifold of dimension $n \geq 3$ without boundary. Given $s\in (0,2)$, $x_0\in M$, and $a\in C^0(M)$, we consider distributional solutions $u\in C^{0}(M)$  to the equation
\begin{equation}\label{art-1-H-S-eqt-1}
\Delta_g u+a(x) u=\frac{u^{\crit-1}}{d_g(x,x_0)^s}\; \; \hbox{ in }M
\end{equation}
where $\crit := \frac{2(n-s)}{n-2}$ is the Hardy-Sobolev exponent. More precisely, let $H_1^2(M)$ be the completion of $C^\infty(M)$ for the norm $u\mapsto \Vert u\Vert_2+\Vert\nabla u\Vert_2$. 
The exponent $\crit$ is critical in the following sense: the Sobolev space $H_1^2(M)$ is continuously embedded in the weighted Lebesgue space $L^p(M, d_g(\cdot, x_0)^{-s})$ if and only if $1\leq p\leq \crit$, 
and this embedding is compact if and only if $1\leq p<\crit$.

\medskip\noindent There is an important literature on  Hardy-Sobolev equations in the Euclidean setting of a domain of $\R^n$, in particular to show existence or non-existence of solutions, 
see for instance Ghoussoub-Yuan \cite{GH-Yuan}, Li-Ruf-Guo-Niu \cite{Li-Ruf-Guo-Niu}, Musina \cite{Musina}, Pucci-Servadei \cite{Pucci-Servadei}, Kang-Peng \cite{Kang-Peng}, and the references therein. 
In particular, in the spirit of Brezis-Nirenberg, Ghoussoub-Yuan \cite{GH-Yuan}
 proved the existence of solution to equations like \eqref{art-1-H-S-eqt-1} when $n\geq 4$ and the potential $a$ achieves negative values at the interior singular point $x_0$. In the present manuscript, 
our objective is both to study the influence of the curvature when dealing with a Riemannian Manifold, 
and to tackle dimension $n=3$.

\medskip\noindent We consider the functional 
$$J(u):=\frac{\int_M(|\nabla u|_g^2+au^2)\, dv_g}{\left(\int_M\frac{|u|^{\crit}}{d_g(x,x_0)^s}\, dv_g\right)^{\frac{2}{\crit}}}\; ; \; u\in H_1^2(M)\setminus \{0\},$$
which is well-defined due to the above-mentioned embeddings. Here $dv_g$ denotes the Riemannian element of volume. When the operator $\Delta_g+a$ is coercive, then, up to multiplication by a positive constant, critical points of the functional $J$ (if they exist) are solutions to equation \eqref{art-1-H-S-eqt-1}. In the sequel, we assume that $\Delta_g+a$ is coercive.
In the spirit of Aubin \cite{Aubin-2}, we investigate the existence of solutions to  \eqref{art-1-H-S-eqt-1} by minimizing the functional $J$: it is classical for this type 
of problem that the difficulty is the lack of compactness for the critical embedding. Since the resolution of the Yamabe problem (see \cite{Aubin-2}, \cite{Schoen-1} and \cite{Tru-2}), it is also well known that there exists a dichotomy 
between high dimension (see Aubin \cite{Aubin-2}) where the arguments are local, and small dimension (see Schoen \cite{Schoen-1}) where the arguments are global.

\medskip\noindent In the sequel, we let $\hbox{Scal}_g(x)$ be the scalar curvature at $x\in M$. We let $G_{x_0}: M\backslash \{x_0\}\to \R$ be the Green's function at $x_0 $ for the operator $\Delta_g+a$ (this is defined since the operator is coercive).
 In dimension $n=3$, there exists $m(x_0)\in\R$ such that for all $\alpha\in (0,1)$
$$G_{x_0}(x)=\frac{1}{\omega_2d_g(x,x_0)}+m(x_0)+O(d_g(x,x_0)^\alpha)\hbox{ when }x\to x_0.$$
Here and in the sequel, $\omega_k$ denote the volume of the canonical $k-$dimensional unit sphere $\mathbb{S}^k$, $k\geq 1$. The quantity $m(x_0)$ is refered to as the mass of the point $x_0\in M$. Our main result states as follows:

\begin{theorem}\label{art-H-S-th-1} Let $x_0\in M$, $s\in (0,2)$, and $a\in C^{0}(M)$ be such that the operator $\Delta_g+a$ is coercive. We assume that
\begin{equation}\label{art-1-H-S-eqt-2}
\left\{\begin{array}{ll}
a(x_0)< c_{n,s}\hbox{Scal}_g(x_0)&\hbox{ if }n\geq 4\\
m(x_0)>0&\hbox{ if }n=3.
\end{array}\right\}
\end{equation} 
with $c_{n,s}:=\frac{(n-2)(6-s)}{12(2n-2-s)}$. Then there exists a positive solution $u\in C^{0}(M)\cap H_1^2(M)$ to the Hardy-Sobolev equation \eqref{art-1-H-S-eqt-1}. 
Moreover, $u\in C^{0,\theta}(M)$ for all $\theta\in (0,\min\{1, 2-s\})$ and we can choose $u$ as a minimizer of $J$.
\end{theorem}

\medskip\noindent As a consequence of the Positive Mass Theorem (see \cite{Schoen-2}, \cite{Schoen-3}), we get (see Druet \cite{Druet-1} and Proposition \ref{art-H-S-cor-ex-1} in Section 4 below) that $m(x_0)>0$ for $n=3$ when $a \leq \hbox{Scal}_g/8$, 
with the additional assumption that $(M,g)$ is not conformally equivalent to the canonical $3-$sphere if $a\equiv \hbox{Scal}_g/8$.

\medskip\noindent Theorem \ref{art-H-S-th-1} suggests some remarks. For equations of scalar curvature type, that is when $s=0$, 
a similar result was obtained by Aubin \cite{Aubin-2} (for $n\geq 4$) and by Schoen \cite{Schoen-1} (see also Druet \cite{Druet-1}) (for $n=3$): however, when $s\in (0,2)$, the problem is subcritical outside the singular point $x_0$, 
and therefore it is natural to get a condition at this point. Another remark is that, when $s=0$, Aubin (see \cite{Aubin-2}) obtained the constant $c_{n,0}$ when $n\geq 4$, the potential $c_{n,0}\hbox{Scal}_g$ 
being such that the Yamabe equation is conformally invariant. When $s\in (0,2)$, the critical equation enjoys no suitable conformal invariance due to the singular term $d_g(\cdot, x_0)^{-s}$, 
and, despite our existence result involves the scalar curvature, one gets another constant $c_{n,s}$.

\medskip\noindent It is also to notice that, unlike the case $s=0$, the solutions to equations like \eqref{art-1-H-S-eqt-1} are not $C^2$. This lead us to handle with care issues related to the maximum principle, 
for which we develop a suitable approach. As in Aubin, the minimization approach leads to computing some test-function estimates. However, unlike the case $s=0$, 
the terms involved in the expansion of the functional are not explicit and we need to collect them suitably to obtain the explicit value of $c_{n,s}$ above.



\medskip\noindent The proof of Theorem \ref{art-H-S-th-1} uses the best constant in the Hardy-Sobolev inequality. It follows from the Hardy-Sobolev embedding that there exist $A,B>0$ such that
\begin{equation}\label{art-H-S-eqt-3}
\left(\int_M \frac{|u|^{\crit}}{d_g(x,x_0)^s}\, dv_g\right)^{\frac{2}{\crit}} \leq A\int_M |\nabla u|_g^2\, dv_g+ B\int_M u^2\, dv_g
\end{equation}
for all $u\in H_1^2(M)$. We let $A_0(M,g,s)$ be the best first constant of the Riemannian Hardy-Sobolev inequality, that is
\begin{equation}\label{art-H-S-eqt-4} 
A_0(M,g,s):= \inf\{A > 0; \eqref{art-H-S-eqt-3} \hbox{ holds for all }u\in H_1^2(M)\}.
\end{equation}
We prove the following:

\begin{theorem}\label{art-H-S-th-2} Let $(M,g)$ be a compact Riemannian Manifold of dimension $n \geq 3$,  
$x_0 \in M$, $s \in (0,2)$ and  $2^\star(s) = \frac{2(n-s)}{n-2}$. Then 
$$A_0(M,g,s) = K(n,s),$$
where $K(n,s)$ is the optimal constant of the Euclidean Hardy-Sobolev inequality, that is
\begin{equation}\label{art-H-S-th-1-eqt-1}
K(n,s)^{-1} := \inf_{\varphi \in C_c^\infty(\R^n)\setminus\{0\}}
\frac{\int_{\R^n}|\nabla\varphi|^2 dX}
{\left(\int_{\R^n}\frac{|\varphi|^{2^\star(s)}}{|X|^s}dX\right)^{\frac{2}{\crit}}}
\end{equation}
\end{theorem}
\noindent Theorem \ref{art-H-S-th-2} was proved by Aubin \cite{Aubin-1} for the case $s=0$. The value of $K(n,s)$ is 
$$K(n,s) = \left[(n-2)(n-s)\right]^{-1}\left(\frac{1}{2-s}\omega_{n-1}\frac{\Gamma^2(n-s/2-s)}{\Gamma(2(n-s)/2-s)}\right)^{-\frac{2-s}{n-s}}.$$
It was computed independently by Aubin \cite{Aubin-1}, Rodemich \cite{Rodemich}  and 
Talenti \cite{Talenti} for the case $s=0$, and the value for $s\in (0,2)$ has been computed by Lieb (see \cite{Lieb-1}, Theorem 4.3). 

\medskip\noindent A natural question is to know whether the infimum $A_0(M,g,s)$ is achieved or not, that is if there exists $B>0$ such that equality \eqref{art-H-S-eqt-3} 
holds for all $u\in H_1^2(M)$ with $A = K(n,s)$. The answer is positive: this is the object of the work \cite{jaber:best:constant}.

\medskip\noindent A very last remark is that Theorem \ref{art-H-S-th-1} holds when $M$ is a compact manifold with boundary provided $x_0$ lies in the interior. In particular, we extend Ghoussoub-Yuan's \cite{GH-Yuan}
 result to dimension $n=3$:
\begin{theorem}\label{art-H-S-th-Dim3} Let $\Omega$ be a smooth bounded domain of $\mathbb{R}^3$ and let $x_0\in\Omega$ be an interior point. For $a\in C^0(\overline{\Omega})$ such that $\Delta+a$ is coercive, 
we define the Robin function as $R(x,y):=\omega_2^{-1}|x-y|^{-1}-G_{x}(y)$ where $G$ is the Green's function for $\Delta+a$ with Dirichlet boundary condition. We assume that $R(x_0,x_0)<0$. 
Then there exists a function $u\in C^{0,\theta}(\overline{\Omega})$ for all $\theta\in (0,\min\{1, 2-s\})$ to the Hardy-Sobolev equation 
$$\Delta u+a(x) u=\frac{u^{\crit-1}}{|x-x_0|^s}\, , u>0\, \; \hbox{ in }\Omega\hbox{ and }u=0\hbox{ on }\partial\Omega.$$
\end{theorem}

\medskip\noindent This paper is organized as follows. In Section 1, we prove Theorem \ref{art-H-S-th-2}. In Section 2, we prove a general existence theorem for solutions to equation (\ref{art-1-H-S-eqt-1}). In Section 3, 
we compute the full expansion of the functional $J$ taken at the relevant test-functions for dimension $n\geq 4$. In Section 4, we perform the test-functions estimate for the specific dimension $n=3$ and prove 
Theorems \ref{art-H-S-th-1}  and \ref{art-H-S-th-Dim3}. 

\medskip\noindent After this work was completed, we learned that Thiam \cite{Elhadji-A.T.} has independently studied similar issues.  
\\ \noindent{\bf Acknowledgments:} The author sincerely thanks Prof. Fr\'ed\'eric Robert for helpful discussions, suggestions and remarks.  

\section{The best constant in the Hardy-Sobolev inequality}

In this section, we will prove Theorem \ref{art-H-S-th-2}. For that, we begin by the following proposition : 

\begin{proposition}\label{art-H-S-prop-1} Let $(M,g)$ be a compact Riemannian Manifold of dimension $n \geq 3$,  
$x_0 \in M$, $s \in (0,2)$. For any $\epsilon > 0$, there exists $B_\epsilon > 0$ such that
\begin{equation}\label{art-H-S-prop-1-eqt-1}
\left(\int_M\frac{|u|^{2^\star(s)}}{d_g(x,x_0)^s}dv_g\right)^{\frac{2}{2^\star(s)}}
 \leq (K(n,s) + \epsilon)\int_M |\nabla u|_g^2 dv_g + B_\epsilon \int_Mu^2dv_g.
\end{equation}
for all $u \in H_1^2(M)$.  
\end{proposition}
Thiam \cite{Elhadji-A.T.} proved a result in the same spirit with addition of an extra remainder term.
The case $s=0$ has been proved by Aubin \cite{Aubin-1}  (see also \cite{Hebey-1}, \cite{Hebey-2} for an exposition in book form). We adapt this proof to our case.

\begin{proof}\par

\noindent{\bf Step 1: Covering of $M$ by geodesic balls.}
For any $x\in M$, we denote as $\hbox{exp}_x$ the exponential map at $x$ with respect to the metric $g$. In the sequel, for any $r>0$ and $z\in M$, $\B_{r}(z)\subset M$ 
denotes the ball of center $0$ and of radius $r$ for the Riemannian distance $d_g$.  For any $x\in M$ and any $\rho > 0$, there exist $r = r(x,\rho)\in (0, i_g(M)/2), \ \lim_{\rho\to 0} r(x,\rho) = 0$ 
(here, $i_g(M)$ denotes the injectivity radius of $(M,g)$) such that the exponential chart $(\B_{2r}(x), \hbox{exp}_{x}^{-1})$ satisfies the following properties: 
on $\B_{2r}(x)$, we have that 
\begin{eqnarray*}
&&(1-\rho)\delta \leq   g \leq (1+\rho)\delta,\\
&&(1-\rho)^{\frac{n}{2}}dx \leq dv_{g}   \leq(1+\rho)^{\frac{n}{2}}dx,\\
&&D_\rho^{-1}|T|_\delta \leq |T|_g \leq D_\rho|T|_\delta, \hbox{ for all } T \in \chi(T^\star M)
\end{eqnarray*}
where $\lim_{\rho\rightarrow+\infty}D_\rho = 1$, $\chi(T^\star M)$ denotes the space of $1-$covariant tensor fields on $M$, $\delta$ is the Euclidean metric on $\R^n$, that is
 the standard scalar product on $\R^n$, and we have assimilated $g$ to the local metric $(\exp_{x_0})^*g$ on $\R^n$ via the exponential map.

\medskip\noindent It follows from the compactness of $M$ that there exists $N\in \N$ (depending on $\rho$) and $x_1,...,x_N\in M$ (depending on $\rho$) such that $$M\setminus\B_{\frac{r_0}{2}}(x_0)\subset \cup_{m=1}^N \B_{r_m}(x_m),$$
where $r_0 = r(x_0, \rho)$ and $r_m = r(x_m, \rho)$.

\medskip\noindent{\bf Step 2:}  We claim that for all $\epsilon > 0$ there exists $\rho_0 = \rho_0(\epsilon) > 0$ such that $\lim_{\epsilon\to0}\rho_0(\epsilon) = 0$ 
and for all $\rho \in (0, \rho_0) $, all $m \in \{0,\ldots,N-1\}$ and all $u \in C_c^\infty(\B_{r_m}(x_m))$, we have that : 
\begin{equation}\label{est:hs:1}
\left(\int_{M}\frac{|u|^{2^\star(s)}}{d_g(x,x_0)^s}dv_g\right)^{\frac{2}{2^\star(s)}} \leq \left(K(n,s) + \frac{\epsilon}{2}\right)\int_{M}|\nabla u|_g^2dv_g.
\end{equation}
Indeed, it follows from \eqref{art-H-S-th-1-eqt-1} that for all $\varphi \in C_c^\infty(\R^n) $ :  
\begin{equation}\label{art-H-S-th-1-eqt-2.1}
\left(\int_{\R^n}\frac{|\varphi|^{2^\star(s)}}{|X|_\delta^s}dX\right)^{\frac{2}{2^\star(s)}}\leq K(n,s)\int_{\R^n}|\nabla\varphi|_{\delta}^2 dX.
\end{equation}
We consider $\rho > 0$, $m \in \{0,\ldots,N\}$ and $u \in C_c^\infty(\B_{r_m}(x_m))$ such that $(\B_{r_m}(x_m), \hbox{exp}_{x_m}^{-1})$ is an exponential card as in Step 1. We distinguish two cases : 
\medskip\noindent 
\\ {\bf Case 2.1 :} If $m = 0$ then using the properties of the exponential card $(\B_{r_0}(x_0), \hbox{exp}_{x_0}^{-1})$, developed in Step 1, and the Euclidean Hardy-Sobolev inequality \eqref{art-H-S-th-1-eqt-2.1}, 
we write 
\begin{eqnarray*}
  \left(\int_M\frac{|u|^{2^\star(s)}}{d_g(x,x_0)^s}dv_g\right)^{\frac{2}{2^\star(s)}}
&\leq& (1+\rho)^{\frac{n}{2^\star(s)}}K(n,s)\int_{\R^n}|\nabla(u\circ\exp_{x_m})|_\delta^2dX
\\ &\leq& D_\rho^2(1+\rho)^{\frac{n}{2^\star(s)}}(1-\rho)^{\frac{-n}{2}}K(n,s)\int_M|\nabla u|_g^2 dv_g.
\end{eqnarray*}
Letting $\rho \to 0$, we get \eqref{est:hs:1}, for all $u \in C_c^\infty(\B_{r_0}(x_0))$, when $m=0$. This proves \eqref{est:hs:1} in the Case 2.1.  

\medskip\noindent{\bf Case 2.2 :} If $m \in  \{1,\ldots,N-1\}$ then for all $x \in \B_{r_m}(x_m)$, we have :
$$d_g(x,x_0) \geq \lambda_0 > 0, $$
with $\lambda_0 = \frac{r_0}{2} - r_m$. Thanks to H\"older inequality and Gagliardo-Nirenberg-Sobolev inequality, we can write that :   
 \begin{eqnarray*}
  \left(\int_M\frac{|u|^{2^\star(s)}}{d_g(x,x_0)^s}dv_g\right)^{\frac{2}{2^\star(s)}} 
 &\leq& \frac{vol(\B_{r_m}(x_m))^{2\left(\frac{1}{2^\star(s)} - \frac{1}{2^\star}\right)}}{\lambda_0^{\frac{2s}{\crit}}}\left(\int_{\B_{r_m}(x_m)}|u|^{2^\star}dv_g\right)^{\frac{2}{2^\star}}
\\ &\leq& Q'_\rho\int_M|\nabla u|_g^2 dv_g,
\end{eqnarray*}
where $\lim_{\rho\to 0}Q'_\rho =0$ and $2^\star:=2n/(n-2)$ is the Sobolev exponent.
Letting $\rho \to 0$, we get \eqref{est:hs:1}, for all $u \in C_c^\infty(\B_{r_m}(x_m))$, when $m \geq 1$. This ends Step 2.    


\medskip\noindent{\bf Step 3:} We fix $\epsilon>0$, $\rho \in (0, \rho_0(\epsilon))$ and $x_1, \ldots, x_N$ as in Step 1 and Step 2.  We consider now $(\alpha_m)_{m=0,\ldots,N-1}$ a $C^\infty$-partition of unity subordinate to
 the covering $(\B_{r_m}(x_m))_{m=0,\ldots,N-1}$ of $M$ and
 define, for all $m= 0 ,\ldots, N-1$, a function $\eta_m$ on $M$ by  $$\eta_m = \frac{\alpha_m^3}{\sum_{i=0}^{N-1}\alpha_i^3}. $$
We can see easily that $(\eta_m)_{m=0, \ldots, N-1}$ is a  $C^\infty$-partition of unity subordinate to the covering $(\B_{r_m}(x_m))_{m=0,\ldots,N-1}$ of $M$ s.t. 
$\eta_m^{\frac{1}{2}} \in C^1(M)$, for every $m = 0, \ldots, N-1$. We let $H > 0$ satisfying for each $m = 0, \ldots, N-1$ : 
\begin{equation}\label{prop-1-eqt-step3}
|\nabla\eta_m^{\frac{1}{2}}|_g \leq H.
\end{equation}

\medskip\noindent{\bf Step 4:} In this step, we will prove the Hardy-Sobolev inequality on $C^\infty(M)$. Indeed, we let $\epsilon > 0$ and $(\eta_m)_{m=0, \ldots, N-1}$ be a $C^\infty$-partition of unity as 
in Step 3 and consider $u \in C^\infty(M)$. Since $\frac{2^\star(s)}{2} > 1$, we get that : 
\begin{eqnarray*}
 &&\left(\int_M\frac{|u|^{2^\star(s)}}{d_g(x,x_0)^s}dv_g\right)^{\frac{2}{2^\star(s)}}\leq  \left(\int_M\frac{|\sum_{m=0}^{N-1}\eta_m
u^2|^{\frac{2^\star(s)}{2}}}{d_g(x,x_0)^s}dv_g\right)^{\frac{2}{2^\star(s)}}\\
&&\leq
\left\Vert\sum_{m=0}^{N-1}\eta_mu^2\right\Vert_{L^{\frac{2^\star(s)}{2}}(M,d_g(x,x_0)^{-s})}
\leq
\sum_{m=0}^{N-1}\left\Vert\eta_mu^2\right\Vert_{L^{\frac{2^\star(s)}{2}}(M,d_g(x,x_0)^{-s})}
 \\ &&\leq \sum_{m=0}^{N-1}\left(\int_M\frac{|\eta_m^{\frac{1}{2}}
u|^{2^\star(s)}}{d_g(x,x_0)^s}dv_g\right)^{\frac{2}{2^\star(s)}}.
 \end{eqnarray*}
Using inequality \eqref{est:hs:1} in Step 2 and by density ($\eta_m^{\frac{1}{2}}u \in C^1(M)$), we get that
 $$ \left(\int_M\frac{|\eta_m^{\frac{1}{2}}u|^{2^\star(s)}}{d_g(x,x_0)^s}dv_g\right)^{\frac{2}{2^\star(s)}}
   \leq  (K(n,s) + \frac{\epsilon}{2})\int_M|\nabla(\eta_m^{\frac{1}{2}}u)|_g^2dv_g. $$
Hence  
\begin{eqnarray*}
   \left(\int_M\frac{|u|^{2^\star(s)}}{d_g(x,x_0)^s}dv_g\right)^{\frac{2}{2^\star(s)}}
   &\leq& (K(n,s) + \frac{\epsilon}{2})\sum_{m=0}^{N-1}\int_M\left(\eta_m|\nabla u|_g^2 +
                  2\eta_m^{\frac{1}{2}}|\nabla u|_g|u||\nabla \eta_m^{\frac{1}{2}}|_g  \right.
 \\ &&\left. \quad + |u|^2|\nabla\eta_m^{\frac{1}{2}}|_g^2\right)dv_g.
   \end{eqnarray*}
Using the Cauchy-Schwarz inequality and \eqref{prop-1-eqt-step3} from Step 3, we get that: 
\begin{equation}\label{art-H-S-th-1-eqt-4.1}
 \left(\int_M\frac{|u|^{2^\star(s)}}{d_g(x,x_0)^s}dv_g\right)^{\frac{2}{2^\star(s)}}
 \leq (K(n,s) + \frac{\epsilon}{2})\left(\|\nabla u\|_2^2 + 2NH\|\nabla u\|_2\|u\|_2 + NH^2\|u\|_2^2\right).
\end{equation}
We choose now $\epsilon_0 > 0$ s.t. 
\begin{equation}\label{art-H-S-th-1-eqt-4.2}
 (K(n,s) + \frac{\epsilon}{2})(1 + \epsilon_0) \leq  K(n,s) + \epsilon.
\end{equation}
Since 
\begin{equation}\label{art-H-S-th-1-eqt-4.3}
2NH\|\nabla u\|_2\|u\|_2 \leq \epsilon_0\|\nabla u\|_2^2 + \frac{(NH)^2}{\epsilon_0}\|u\|_2^2,
\end{equation}   
then by combining \eqref{art-H-S-th-1-eqt-4.1} with \eqref{art-H-S-th-1-eqt-4.2} and \eqref{art-H-S-th-1-eqt-4.3}, we get that : 
\begin{eqnarray*}\label{art-H-S-th-1-eqt-4.4}
 \left(\int_M\frac{|u|^{2^\star(s)}}{d_g(x,x_0)^s}dv_g\right)^{\frac{2}{2^\star(s)}}
 &\leq& (K(n,s) + \epsilon)\int_M|\nabla u|_g^2dv_g + B_\epsilon\int_M|u|^2dv_g,\nonumber
\\ && 
\end{eqnarray*}
where $B_\epsilon =  (\frac{(NH)^2}{\epsilon_0} + NH^2)(K(n,s) + \frac{\epsilon}{2})$. This proves inequality \eqref{art-H-S-prop-1-eqt-1} for functions $u\in C^\infty(M)$. The inequality for $H_1^2(M)$ follows by density. 
This ends the proof of Proposition \ref{art-H-S-prop-1}.
\end{proof}

\medskip\noindent {\bf Proof of Theorem \ref{art-H-S-th-2}}:  We let $A\in \R$ be such that there exists $B>0$ such that inequality \eqref{art-H-S-eqt-3} holds for all $u\in H_1^2(M)$. Therefore, we have that 
\begin{equation}\label{proof-A0-equal-K(n,s)-eqt-1}
\left(\int_M\frac{|u|^{2^\star(s)}}{d_g(x,x_0)^s}dv_g\right)^{\frac{2}{2^\star(s)}}\leq A\int_M|\nabla u|_g^2dv_g + B\int_Mu^2dv_g.
\end{equation} 
We consider $\phi \in C_c^\infty(\R^n)$ such that $Supp\, \phi \subset \B_R(0)$, $R > 0$ and $(\B_{\rho_0}(x_0), \exp_{x_0}^{-1})$ an exponential chart centered at $x_0$ with $\rho_0\in (0, i_g(M))$. For all  $\mu>0$ sufficiently small $(\mu \leq \frac{\rho_0}{R})$, 
we let $\phi_\mu\in C^\infty(\B_{\rho_0}(x_0))$ be such that
$$\phi_\mu(x) = \phi(\mu^{-1}\exp_{x_0}^{-1}(x)) $$
for all $x \in \B_{\rho_0}(x_0)$. Applying, by density,  \eqref{proof-A0-equal-K(n,s)-eqt-1} to $\phi_\mu$, we  write : 
\begin{equation}\label{proof-A0-equal-K(n,s)-eqt-2}
\left(\int_{\B_{\mu R}(x_0)}\frac{|\phi_\mu|^{2^\star(s)}}{d_g(x,x_0)^s}dv_g\right)^{\frac{2}{2^\star(s)}} \leq A\int_{\B_{\mu R}(x_0)}|\nabla\phi_\mu|_g^2dv_g + B\int_{\B_{\mu R}(x_0)}\phi_\mu^2dv_g.
\end{equation} 
For all $\epsilon > 0$, there exists $R_\epsilon > 0$ such that
$$(1-\epsilon)\delta \leq g \leq (1+\epsilon)\delta $$
in $\B_{R_\epsilon}(x_0)$, where $g$ is assimilated to the local metric $(\exp_{x_0})^*g$ on $\R^n$. Then, for all $\mu > 0$ sufficiently small such that $R\mu<R_\epsilon$, we get successively that : 

\begin{equation}\label{proof-A0-equal-K(n,s)-eqt-3}
\int_{\B_{\mu R}(x_0)}\frac{|\phi_\mu|^{2^\star(s)}}{d_g(x,x_0)^s}dv_g \geq  (1-\epsilon)^{\frac{n}{2}}\mu^{n-s}\int_{\B_R(0)}\frac{\phi^{2^\star(s)}(X)}{|X|^s}dX, 
\end{equation}
\begin{equation}\label{proof-A0-equal-K(n,s)-eqt-4}
\int_{\B_{\mu R}(x_0)}|\nabla\phi_\mu|_g^2dv_g \leq (1+\epsilon)^{\frac{n}{2}+1}\mu^{n-2}\int_{\B_R(0)}|\nabla\phi|_\delta^2dX
\end{equation}
and
\begin{equation}\label{proof-A0-equal-K(n,s)-eqt-5}
\int_{\B_{\mu R}(x_0)}\phi_\mu^2dv_g \leq (1+\epsilon)^\frac{n}{2}\mu^n\int_{\B_R(0)}\phi^2dX.
\end{equation} 
Plugging the estimates \eqref{proof-A0-equal-K(n,s)-eqt-3}, \eqref{proof-A0-equal-K(n,s)-eqt-4} and \eqref{proof-A0-equal-K(n,s)-eqt-5} into  \eqref{proof-A0-equal-K(n,s)-eqt-2}, letting $\mu\to 0$ and then $\epsilon\to 0$, we get that
$$\left(\int_{\R^n}\frac{\phi^{2^\star(s)}(X)}{|X|^s}dX\right)^{\frac{2}{2^\star(s)}}  \leq A\int_{\R^n}|\nabla\phi|_\delta^2dX, \ \hbox{ for all }\phi\in C^\infty_c(\R^n).$$
It then follows from the definition of $K(n,s)$ that $A \geq K(n,s)$. Therefore, it follows from the definition of $A_0(M,g,s)$ that $A_0(M,g,s)\geq K(n,s)$. By Proposition \ref{art-H-S-prop-1}, we have that $A_0(M,g,s) \leq K(n,s)$. 
Therefore, $A_0(M,g,s) = K(n,s)$. This proves Theorem \ref{art-H-S-th-2}.
\medskip\noindent
\\ {\bf Remark:} Proposition \ref{art-H-S-prop-1} does not allow to conclude whether $A_0(M,g,s)$ is achieved or not, that is of one can take $\epsilon=0$ in \eqref{art-H-S-prop-1-eqt-1}. Indeed, in our construction, when $\epsilon \to 0$, $r_m \to 0$ 
and then $H \geq |\nabla\eta_m^{\frac{1}{2}}|_g \to +\infty$ (see the proof of Proposition \ref{art-H-S-prop-1}). This implies that $\lim_{\epsilon\to0}B_\epsilon = +\infty$. Proving that $A_0(M,g,s)$ is achieved 
required different techniques and blow-up analysis: this is the object of the article \cite{jaber:best:constant}.

\section{A general existence theorem}

This section is devoted to the proof of the following Theorem:
\begin{theorem}\label{art-H-S-th-3} Let $(M,g)$ be a compact Riemannian Manifold of dimension $n\geq 3$ without boundary. We fix $s\in (0,2)$, $x_0\in M$, and $a\in C^0(M)$ such that $\Delta_g+a$ is coercive. We assume that
\begin{equation}\label{art-1-H-S-th-3-meancdt}
\inf_{u\in H_1^2(M)\setminus \{0\}}J(u)<\frac{1}{K(n,s)}
\end{equation}
Then the infimum of $J$ on $H_1^2(M)\setminus \{0\}$ is achieved by a positive function $u\in H_1^2(M)\cap C^0(M)$. 
Moreover, up to homothety, $u$ is a solution to \eqref{art-1-H-S-eqt-1} and 
$u\in C^{0,\theta}(M) \cap C^{1,\alpha}_{loc}(M\setminus\{x_0\})$ for all $\theta\in (0,\min\{1, 2-s\})$ and $\alpha \in (0,1)$.
\end{theorem}
\noindent The existence of a minimizer of $J$ in $H_1^2(M)\setminus\{0\}$ has been proved independently by Thiam \cite{Elhadji-A.T.}.

\medskip\noindent We prove Theorem \ref{art-H-S-th-3} via the classical subcritical approach. For any $q \in (2,\crit]$, we define 
$$J_q(u):=\frac{\int_M(|\nabla u|_g^2+au^2)\, dv_g}{\left(\int_M |u|^{q}\, dv_g\right)^{\frac{2}{q}}}\; ; \; u\in H_1^2(M),$$
and
$$\mathcal{H}_q = \left\{ u \in H_1^2(M) \; ; \; \int_{M}\frac{|u|^q}{d_g(x,x_0)^s}dv_g = 1\right\}. $$
Finally, we define:
$$\lambda_q = \inf_{u\in H_1^2(M)\setminus\{0\}}J_q(u).$$
We fix $q\in (2,\crit)$. Since the embedding $H_1^2(M)\hookrightarrow L^q(M, d_g(\cdot, x_0)^{-s})$ is compact, there exists a minimizer for $\lambda^{(s)}_q$. More precisely,  there exists $u_q \in H_1^2(M)\setminus\{0\} \bigcap \mathcal{H}_q $, $u_q \geq 0$ a.e. such that $u_q$ verifies weakly the subcritical Hardy-Sobolev equation :
$$ \Delta_g u_q + au_q = \lambda_q\frac{u_q^{q-1}}{d_g(x,x_0)^s}\hbox{ in }M. $$
In particular, we have that $\lambda_q = J_q(u_q)$. 

\medskip\noindent Now we proceed in several steps.

\medskip\noindent{\bf Step 1:} We claim that the sequence $(\lambda_q)_q$ converge to $\lambda_{\crit}$ when $q\rightarrow \crit$. 

\smallskip\noindent The proof follows the standard method described in \cite{Yamabe} and \cite{Aubin-2} for instance. We omit the proof.


\medskip\noindent{\bf Step 2:} As one checks, the sequence $(u_q)_q$ is bounded in $H_1^2(M)$ independently of $q$. Therefore, there exists $u \in H_1^2(M)$, $u \geq 0$ a.e. such that, 
up to a subsequence, $(u_q)_q$ converge to $u$ weakly in $H_1^2(M)$ and strongly in $L^2(M)$, moreover, the convergence holds a.e. in $M$. It is classical (see \cite{Yamabe} and \cite{Aubin-2}) that $u\in H_1^2(M)$ is a weak solution to 
\begin{equation*}
\Delta_g u + au = \lambda_{\crit} \frac{u^{\crit-1}}{d_g(x,x_0)^s}\hbox{ in }M\, ;\, u\geq 0\hbox{ a.e. in }M.
\end{equation*}

\medskip\noindent{\bf Step 3:} We claim that $u\not\equiv 0$ is a minimizer of $J^{(s)}$ and that $(u_q)_q \to u$ strongly in $H_1^2(M)$.

\smallskip\noindent Indeed, it follows from the hypothesis \eqref{art-1-H-S-th-3-meancdt} that there exists $\epsilon_0 > 0$ such that 
\begin{equation}\label{art-H-S-th-3-eqt-4}
\lambda_{\crit}(K(n,s) + \epsilon_0) < 1.
\end{equation}
Now from Proposition \ref{art-H-S-prop-1}, we know that there exists $B_{\epsilon_0} > 0$ such that for all $q \in (2,\crit)$ : 
\begin{equation}\label{art-H-S-th-3-eqt-5}
\left(\int_{M}\frac{|u_q|^{2^\star(s)}}{d_g(x,x_0)^s}dv_g\right)^{\frac{2}{\crit}} \leq (K(n,s) + \epsilon_0)\int_M|\nabla u_q|^2dv_g + B_{\epsilon_0}\int_Mu_q^2dv_g.
\end{equation}
H\"older inequality and $u_q\in H_q$ yield:
\begin{equation}\label{art-H-S-th-3-eqt-6}
\left(\int_{M}\frac{|u_q|^{2^\star(s)}}{d_g(x,x_0)^s}dv_g\right)^{\frac{2}{\crit}} \geq 1.
\end{equation}
Combining \eqref{art-H-S-th-3-eqt-5} and \eqref{art-H-S-th-3-eqt-6}, we get : 
$$\|1\|_{q',s}\left[(K(n,s) + \epsilon_0)\lambda_q + B_{\epsilon_0}\int_Mu_q^2dv_g\right] \geq 1, $$
where $q' > 1$ verifies $q^{-1} = (q')^{-1} + (\crit)^{-1}$.
Letting $q\rightarrow\crit$ in the last relation, we  write : 
$$(K(n,s) + \epsilon_0)\lambda_{\crit} + B(\epsilon_0)\int_Mu^2dv_g \geq 1. $$
It then follows from \eqref{art-H-S-th-3-eqt-4} that $B_{\epsilon_0}\int_Mu^2dv_g>0$, and then $u \not\equiv 0$. It is then classical that $u\in H_1^2(M)$ is a minimizer and that $u_q\to u$ strongly in $H_1^2(M)$ when $q\to \crit$.

\medskip\noindent{\bf Step 4}: We claim that $u \in C^{0,\theta}(M)$, for all $\theta \in (0, \min\{1, 2-s\})$. Following the method used 
in \cite{GH-R1} (see Proposition 8.1) inspired from the strategy developed by Trudinger \cite{Tru-2} for the Yamabe problem, we get that $u \in L^p(M)$, 
for all $p \geq 1$. Defining $f_u(x):= \frac{u(x)^{\crit-1}}{d_g(x,x_0)^s}$, we then get from H\"older inequality that $f_u \in L^p(M)$, for all 
$p \in [1, \frac{n}{s})$. Since $\Delta_gu + au = f_u$ and $u \in H_1^2(M)$ and $s\in (0,2)$, it follows from standard elliptic theory (see \cite{Gil-Tru}) that
$u \in C^{0,\theta}(M)$, for all $\theta \in (0, \min\{1, 2-s\})$. 

\medskip\noindent{\bf Step 5}:  We claim that $u \in C^{1,\alpha}_{loc}(M\setminus\{x_0\})$, for all $\alpha\in (0,1)$. Indeed, since $u\in L^p(M)$ for all $p>1$ (see Step 4), we get that $f_u\in L^p_{loc}(M\setminus\{x_0\})$ for all $p>1$.
Since $\Delta_gu + au = f_u$ and $u \in H_1^2(M)$, then, up to taking $p > n$ sufficiently large, it follows from standard elliptic theory (see \cite{Gil-Tru}) that $u \in C^{1,\alpha}_{loc}(M\setminus\{x_0\})$ for all $\alpha\in (0,1)$.\par
\smallskip\noindent{\it Remark:} If $a \in C^{0,\gamma}(M)$ for some $\gamma\in (0,1)$ then, using the same argument as above, we get that $u \in C^{2,\gamma}_{loc}(M\setminus\{x_0\})$.

\medskip\noindent{\bf Step 6:} \ We claim that $u > 0$ on $M$. Indeed, we consider $x_1 \neq x_0$ such that $\B_{2r}(x_1) \subset \subset M\setminus\{x_0\}$, 
with $r > 0$ sufficiently small and a function $h$ defined on $\B_{2r}(x_1)$ by $h(x) := a(x) - \lambda_{\crit}\frac{|u(x)|^{\crit-2}}{d_g(x,x_0)^s}$.
Clearly, we have that $h \in C^0(\overline{\B_{2r}(x_1)})$. Since $u \in H_1^2(\B_{2r}(x_1))$, $u \geq 0$ and $(\Delta_g + h)u = 0$ on $\B_{2r}(x_1)$. It then follows from standard elliptic theory (see \cite{Gil-Tru}, Theorem 8.20) that
there exists $C = C(M,g,x_1,r) > 0$ such that $\sup_{\B_r(x_1)}u \leq C\inf_{\B_r(x_1)}u$. This implies that $u_{|\B_r(x_1)} > 0$. Therefore, $u(x)>0$ for all $x\in M\setminus\{x_0\}$. 

\medskip\noindent We are left with proving that $u(x_0) > 0$. We argue by contradiction and we assume that $u(x_0) = 0$.
 
\smallskip\noindent{\bfseries Step 6.1.:} \ We claim that $u$ is differentiable at $x_0$.  Here again, we follow the method used in \cite{GH-R1} (see Proposition 8.1). 
Since $u \in C^{0,\alpha}(M)$, for all $\alpha \in (0, \min\{1, 2-s\})$ (from Step 4) and $u(x_0) = 0$ then for any $\alpha \in (0, \min\{1, 2-s\})$, there
exists a constant  $C_1(\alpha) = C(M, g, \alpha) > 0$ such that 
\begin{equation}\label{General-result-H-S-Step-6-eqt-1}
|u(x)| \leq C_1(\alpha)d_g(x,x_0)^\alpha
\end{equation}
for all $x\in M$. Therefore, we have that
\begin{equation}\label{eq:u:fu}
\Delta_gu + au = f_u ,
\end{equation}
where with \eqref{General-result-H-S-Step-6-eqt-1}, we have that
\begin{equation}\label{General-result-H-S-Step-6-eqt-2}
|f_u(x)| \leq \frac{C_2(\alpha)}{d_g(x,x_0)^{s-\alpha(\crit-1)}} 
\end{equation}
for all $x\in M\setminus\{x_0\}$.

\medskip\noindent We claim that $u \in C^{0,\alpha}(M)$, for all $\alpha \in (0,1)$. 

\smallskip\noindent Indeed, we define $\alpha_1 := \sup\{\alpha \in (0, 1)\ ; \ u \in C^{0,\alpha}(M)\}$ and $N'_s = s - \alpha_1(\crit-1)$ and distinguish the following cases : 
\\ \\ $\bullet$ {\itshape Case 6.1.1} \ \ $N'_s \leq 0$. In this case, up to taking $\alpha$ close enough 
to $\alpha_1$, we get  that $f_u \in L^p(M)$, for all $p \geq 1$. It follows from \eqref{eq:u:fu} and standard
elliptic theory that there exists $\theta \in (0,1)$ such that $u\in C^{1,\theta}(M)$. This proves that $\alpha_1 = 1$ in Case 6.1.1.
\\ $\bullet$ {\itshape Case 6.1.2} \ \ $0 < N'_s < 1$. In this case, up to taking $\alpha$ close enough to $\alpha_1$, we get  that $f_u \in L^p(M)$, for all $p < \frac{n}{N'_s}$.  Since $1 > N'_s$
 then there exists $p \in (n, \frac{n}{N'_s})$ such that $f_u \in L^p(M)$. Therefore, \eqref{eq:u:fu} and standard
elliptic theory yield the existence of $\theta \in (0,1)$ such that $u \in C^{1,\theta}(M)$. This proves that $\alpha_1 = 1$ in Case 6.1.2.
\\ $\bullet$ {\itshape Case 6.1.3} \ \ $N'_s = 1$. In this case, up to taking $\alpha$ close enough to $\alpha_1$, we get  that $f_u \in L^p(M)$, for all $p < n$.  
This implies that for any $p \in (\frac{n}{2}, n)$, we have that $f_u \in L^p(M)$. Equation \eqref{eq:u:fu} and standard
elliptic theory then yields $u\in C^{0,\theta}(M)$ for all $\theta\in (0,1)$. This proves that $\alpha_1 = 1$ in Case 6.1.3.
\\ $\bullet$ {\itshape Case 6.1.4} \ \ $N'_s > 1$. In this case, up to taking $\alpha$ close enough to $\alpha_1$, we get  that $f_u \in L^p(M)$, for all $p < \frac{n}{N'_s}$. 
Therefore, \eqref{eq:u:fu}, $N'_s \in (1,2)$ (because $N'_s > 0$ and $s < 2$), and standard elliptic theory yield $u \in C^{0,\theta}(M)$ 
for all $\theta < 2 - N'_s$. It then follows from the definition of $\alpha_1$ that $\alpha_1 \geq  2 - N'_s$. This leads to a contradiction with the definition of $N'_s$. Then Case 6.1.4 does not occur. 

\smallskip\noindent These four cases imply that $u \in C^{0,\alpha}(M)$, for all $\alpha \in (0,1)$. This proves the claim.

\medskip\noindent In order to end Step 6.1, we proceed as the above, let $N''_s = s - \crit + 1$ and distinguish two cases : 
\\ $\bullet$ {\itshape Case 6.1.5} \ \ $N''_s \leq 0$. In this case, up to taking $\alpha$ close enough 
to $1$, we have that $f_u \in L^p(M)$, for all $p \geq 1$. Therefore, \eqref{eq:u:fu} and elliptic theory yield $u \in C^1(M)$. This proves Step 6.1 in Case 6.1.5. 
\\ $\bullet$ {\itshape Case 6.1.6} \ \ $N''_s > 0$. In this case, up to taking $\alpha$ close enough 
to $1$, we have that $f_u \in L^p(M)$ for all $p < \frac{n}{N''_s}$. Since $1 > N''_s$, there exists $p \in (n, \frac{n}{N''_s})$ such that 
$f_u \in L^p(M)$. Therefore, it follows from \eqref{eq:u:fu} and elliptic theory that $u \in C^{1}(M)$. This proves the claim of Step 6.1 in Case 6.1.6.

\smallskip\noindent This ends Step 6.1.

\medskip\noindent {\bfseries Step 6.2:} We prove the contradiction here. Since $u\in C^1(M)$, we are able to follow the strategy of \cite{Gil-Tru} (see Lemma 3.4) to adapt Hopf's strong maximum principle. 
We let $\Omega \subset M\setminus\{x_0\}$ be an open set such that $x_0 \in \partial\Omega$ and $\partial\Omega$ satisfies an interior sphere condition at $x_0$,
then there exists an exponential chart $(\B_{2r_y}(y), \exp_y^{-1})$, $y \in \Omega, r_y > 0$ small enough such that $\B_{r_y}(y) \cap \partial\Omega = \{x_0\}$. We consider $C > 0$ such that
$$L_{g,C}(-u) := -(\Delta_g + C)(-u) \geq (\Delta_g + a)(u) \geq 0 $$
on $\Omega$. We fix $\rho \in (0,{r_y})$ and introduce the function $v_\rho$ defined on the annulus $\B_{r_y}(y)\setminus\B_\rho(y)$ by $v_\rho(x) = e^{-k r^2} - e^{-k r_y^2} $ where $r := d_g(x,y)$ and $k > 0$ to be determined. Now, if $\lambda(x)$ is 
the smaller eigenvalue of $g^{-1}$ then that for any $x \in \B_{r_y}(y)\setminus\B_\rho(y)$ we have that:
$$L_{g,C}v_\rho(x) \geq e^{-k r^2}\left[4k^2\lambda(x)r^2 - 2k\left(\sum_{i=1}^ng^{ii} + \Gamma_0r\right) - C\right]  $$ 
where $\Gamma_0 = \Gamma_0(g)$. Hence we choose $k$ large enough so that $L_{g,C}v_\rho \geq 0$ on $\B_{r_y}(y)\setminus\B_\rho(y)$. 
Since $-u < 0$ on $\partial\B_\rho(y)$ then there exists a constant $\epsilon > 0$ such that $- u + \epsilon v_\rho \leq 0$ on $\partial\B_\rho(y)$. Thus we have 
$- u + \epsilon v_\rho \in H_1^2(\B_{r_y}(y)\setminus\B_\rho(y))$, $- u + \epsilon v_\rho \leq 0$ on $\partial\B_\rho(y)$ and $L_{g,C}(- u + \epsilon v_\rho) \geq 0$ on $\B_{r_y}(y)\setminus\B_\rho(y)$. 
It follows from the weak maximum principle (see Theorem 8.1 in \cite{Gil-Tru}) that
\begin{equation}\label{General-result-H-S-Step-6-eqt-3}
- u + \epsilon v_\rho \leq 0, \ \ {\rm on} \ \B_{r_y}(y)\setminus\B_\rho(y)
\end{equation}
\\ In the sequel, $\B_r(0)$ denotes a ball in $(\R^n, \delta)$ centered at the origin and of radius $r$. Now we define $\tilde{u} = u \circ\exp_y$ and $\tilde{v}_\rho = v \circ\exp_y$ on $\B_{r_y}(0)$. By \eqref{General-result-H-S-Step-6-eqt-3}, we get : 
\begin{equation}\label{General-result-H-S-Step-6-eqt-4}
 \epsilon\tilde{v}_\rho \leq \tilde{u}, \ \ {\rm on} \ \B_{r_y}(0)\setminus\B_\rho(0)
\end{equation}
We define $X_0 := \exp_y^{-1}(x_0)$. Since $\hat{u}(X_0) = \hat{v}_\rho(X_0) = 0$, then, by \eqref{General-result-H-S-Step-6-eqt-4}, we can write that
 $$\frac{\partial\hat{u}}{\partial\nu}(X_0) := \liminf_{t \xrightarrow{t<0} 0}\frac{\tilde{u}(X_0 + t\nu) - \tilde{u}(X_0)}{t} \leq \epsilon\liminf_{t\xrightarrow{t<0} 0}
\frac{\tilde{v}_\rho(X_0 + t\nu) - \tilde{v}_\rho(X_0)}{t} := \epsilon\frac{\partial\hat{v}_\rho}{\partial\nu}(X_0),$$
where $\nu$ is the outer normal vector field on $\B_{r_y}(y)$.
Therefore $\frac{\partial\hat{u}}{\partial\nu}(X_0) \leq \epsilon\frac{\partial\tilde{v}_\rho}{\partial\nu}(X_0)$, but $\frac{\partial \tilde{v}_\rho}{\partial\nu}(x_0) = v'_\rho(R) $, 
it follows that $$\frac{\partial\hat{u}}{\partial\nu}(X_0) \leq \epsilon v'_\rho(r_y) < 0.$$
This is a contradiction since $\min_M u=u(x_0)$ and therefore  $\nabla \tilde{u}(X_0) = \nabla u(x_0) = 0$. This ends the proof of Step 6.  

\section{Test-functions for $n\geq 4$}
We consider the test-function sequence $(u_\epsilon)_{\epsilon>0}$ defined, for any $\epsilon > 0, x \in M$, by 
\begin{equation}\label{art-H-S-fct-test-p1-eqt-1}
u_\epsilon(x) = \left(\frac{\epsilon^{1-\frac{s}{2}}}{\epsilon^{2-s} + d_g(x,x_0)^{2-s}}\right)^{\frac{n-2}{2-s}}, 
\end{equation}
the function $\Phi$ defined on $\R^n$ by
\begin{equation}\label{art-H-S-fct-test-p1-eqt-1.1}
\Phi(X) = (1 + |X|^{2-s})^{-\frac{n-2}{2-s}}.
\end{equation}  
Since $u_\epsilon$ is a Lipschitz function, we have that $u_\epsilon \in H_1^2(M)$, for any $\epsilon > 0$. Given $\rho \in (0, i_g(M))$, where $i_g(M)$ is the injectivity radius on $M$, 
we recall that $\B_\rho(x_0)$ be the geodesic ball of center $x_0$ and radius $\rho$. Cartan's expansion of the metric $g$ (see \cite{Lee-parker}) in the exponential chart $(\B_\rho(x_0), exp_{x_0}^{-1})$  yields
\begin{equation}\label{art-H-S-fct-test-p1-eqt-2}
det(g)(x) = 1 - \frac{R_{\alpha\beta}}{3}(x_0)x^\alpha x^\beta + O(r^3), 
\end{equation}
where the $x^\alpha$'s are the coordinates of $x$, $r^2 = \sum_\alpha(x^\alpha)^2$  and $(R_{\alpha\beta})$ is the Ricci curvature. Integrating on the unit sphere $\mathbb{S}^{n-1}$ yields
$$\int_{\mathbb{S}^{n-1}}\sqrt{det(g)}(r\theta)d\theta = \omega_{n-1}\left[1-\frac{Scal_g(x_0)}{6n}r^2 +O(r^3)\right]. $$

\subsection{Estimate of the gradient term.}
At first, we estimate $\int_M|\nabla u_\epsilon|_g^2dv_g$. For that, we write for all $x \in M$ :
$$|\nabla u_\epsilon|_g^2(x) = (n-2)^2\epsilon^{n-2}\frac{r^{2(1-s)}}{(\epsilon^{2-s} + r^{2-s})^{\frac{2(n-s)}{2-s}}} $$
where $r=d_g(x,x_0)$. 
Therefore, using \eqref{art-H-S-fct-test-p1-eqt-1} and the change of variable $t = r\epsilon^{-1}$,  we get that
\begin{eqnarray}\label{art-H-S-fct-test-p1-eqt-3}
\int_{\mathbb{B}_\rho(x_0)}|\nabla u_\epsilon|_g^2dv_g &=&
(n-2)^2\epsilon^{n-2}\omega_{n-1}\times\int_0^\rho
\frac{r^{n+1}\left(1-\frac{Scal_g(x_0)}{6n}r^2 + O(r^3)\right)dr}{r^{2s}(\epsilon^{2-s} +r^{2-s})^{\frac{2(n-s)}{2-s}}} \nonumber
\\ &=& (n-2)^2\omega_{n-1}\int_0^{\frac{\rho}{\epsilon}}\frac{t^{n+1}(1-\frac{Scal_g(x_0)}{6n}(\epsilon t)^2
+ O((\epsilon t)^3))dt}{t^{2s}(1 +t^{2-s})^{\frac{2(n-s)}{2-s}}}.
\end{eqnarray}
Straightforward computations yield
\begin{equation}\label{art-H-S-fct-test-p1-eqt-4}
\int_0^{+\infty}\frac{t^{n+1}dt}{t^{2s}(1 + t^{2-s})^{\frac{2(n-s)}{2-s}}} = (n-2)^{-2}\omega_{n-1}^{-1} \int_{\R^n}|\nabla\Phi|^2dX,
\end{equation}
and 
\begin{equation}\label{art-H-S-fct-test-p1-eqt-5}
\epsilon^2\int_0^{+\infty}\frac{t^{n+1}dt}{t^{2s}(1 + t^{2-s})^{\frac{2(n-s)}{2-s}}}   = \left\{
                \begin{array}{ll}
             \epsilon^2(n-2)^{-2}\omega_{n-1}^{-1} \int_{\R^n}|X|^2|\nabla\Phi|^2dX   & \hbox{ if } \; n \geq 5 \hbox{ ,}\\
         \epsilon^2\ln\frac{1}{\epsilon}& \hbox{ if } \; n = 4 \hbox{ ,}\\
        O(\epsilon) & \hbox{ if } \; n = 3
\hbox{.}\end{array}\right.
\end{equation}
Since
$$ \int_{M\setminus\mathbb{B}_\rho(x_0)}|\nabla u_\epsilon|_g^2dv_g = O(\epsilon^{n-2}),$$
when $\epsilon\to 0$, putting together \eqref{art-H-S-fct-test-p1-eqt-3} with \eqref{art-H-S-fct-test-p1-eqt-4} and \eqref{art-H-S-fct-test-p1-eqt-5} yield
\begin{equation}\label{art-H-S-fct-test-p1-eqt-6}
 \int_M|\nabla u_\epsilon|^2dv_g  = \left\{
                \begin{array}{ll}
               \int_{\R^n}|\nabla\Phi|^2dX  - \frac{\int_{\R^n}|X|^2|\nabla\Phi|^2dX}{6n} \hbox{Scal}_g(x_0)\epsilon^2 +
o(\epsilon^2) & {\rm if} \; n \geq 5 \hbox{ ,} \\
&\\
                 \int_{\R^n}|\nabla\Phi|^2dX - \frac{\omega_3}{6}\hbox{Scal}_g(x_0)\epsilon^2ln(\frac{1}{\epsilon}) + O(\epsilon^2) &
{\rm if} \; n = 4 \hbox{ ,}\\
& \\
                 \int_{\R^n}|\nabla\Phi|^2dX  + O(\epsilon) & {\rm if} \; n = 3
\hbox{,}\\
               \end{array}
              \right.
\end{equation}
\medskip\noindent 
\medskip\noindent Arguing as the above and using that $a\in C^0(M)$, we get that : 
\begin{equation}\label{art-H-S-fct-test-p1-eqt-7}
\int_Mau_\epsilon^2dv_g   = \left\{
                \begin{array}{ll}
               \epsilon^2a(x_0)\int_{\R^n}\Phi^2dX + o(\epsilon^2) & \hbox{ if } \; n
\geq 5 \hbox{ ,} \\
&
\\
                a(x_0)\omega_3\epsilon^2\ln\frac{1}{\epsilon} + O(\epsilon^2) & \hbox{ if } \; n = 4 \hbox{ ,}\\
 & \\
                 O(\epsilon) & {\rm si} \; n = 3 \hbox{,}
               \end{array}
              \right.
\end{equation}
and 
\begin{equation}\label{art-H-S-fct-test-p1-eqt-8}
\int_M\frac{|u_\epsilon|^{\crit}}{d_g(x,x_0)^s}dv_g  = \left\{
                \begin{array}{ll}
               \int_{\R^n}\frac{|\Phi|^{\crit}}{|X|^s}dX -\epsilon^2\frac{Scal_g(x_0)}{6n}\int_{\R^n}|X|^{2-s}|\Phi|^{\crit}dX
 + o(\epsilon^2)  &{\rm if} \; n \geq 4  \hbox{ ,} 
\\ \int_{\R^3}\frac{|\Phi|^{\crit}}{|X|^s}dX + O(\epsilon)   &{\rm if} \; n =3  \hbox{.} 
   \end{array}
     \right.
\end{equation}
From Lieb \cite{Lieb-1}, we know that $\Phi$ is an extremal for \eqref{art-H-S-th-1-eqt-1}, that is
\begin{equation}\label{art-H-S-fct-test-p1-Lieb-1}
\frac{\int_{\R^n}|\nabla \Phi|^2dX}{\left(\int_{\R^n}\frac{|\Phi|^{\crit}}{|X|^s}dX\right)^{\frac{2}{\crit}}} = \frac{1}{K(n,s)}
\end{equation}
Combining \eqref{art-H-S-fct-test-p1-eqt-6}, \eqref{art-H-S-fct-test-p1-eqt-7} and \eqref{art-H-S-fct-test-p1-eqt-8} and this last equation, we obtain, for any $\epsilon > 0$, the following results :

\begin{equation}\label{art-H-S-fct-test-p1-eqt-9}
J(u_\epsilon) =  \frac{1}{K(n,s)}\left(1+
\left\{
\begin{array}{ll}
\left(C_1(n,s) a(x_0)-C_2(n,s)\hbox{Scal}_g(x_0)\right)\epsilon^2+o(\epsilon^2) &\hbox{ if }n\geq 5\\
  \omega_3(\int_{\R^4}|\nabla\Phi|^2\, dX)^{-1}\left(a(x_0) - 
\frac{1}{6}Scal_g(x_0)\right)\epsilon^2ln(\frac{1}{\epsilon}) + O(\epsilon^2)&\hbox{ if }n=4\\
O(\epsilon)& \hbox{ if }n=3
\end{array}\right\}\right)
\end{equation}
where
\begin{eqnarray*}
C_1(n,s) &:= &\frac{\int_{\R^n}|\Phi|^2dX}{\int_{\R^n}|\nabla\Phi|^2dX}\\
C_2(n,s) &:= &\frac{2}{\crit 6 n}\frac{\int_{\R^n}|X|^{2-s}|\Phi|^{\crit}dX}{\int_{\R^n}\frac{|\Phi|^{\crit}}{|X|^s}dX}-\frac{1}{6n}\frac{\int_{\R^n}|X|^2|\nabla\Phi|^2dX}{\int_{\R^n}|\nabla\Phi|^2dX}
\end{eqnarray*}
Unlike the case $s=0$, it is not possible to compute explicitly the constants $C_1(n,s)$ and $C_2(n,s)$. However, we are able to explicit their quotient, which is enough to prove our theorem. We need the following lemma taken from Aubin \cite{Aubin-2} :
\begin{lemma}\label{art-H-S-lem-3}
Let $p,q \in \R_+^*$ such that $p-q > 1$ and assume that $I_p^q = \int_0^{+\infty}\frac{t^qdt}{(1+t)^p}$,
then $$I_{p+1}^q = \frac{p-q-1}{p}I_p^q  \hbox{ and }\, \\ I_{p+1}^{q+1} = \frac{q+1}{p-q-1}I_{p+1}^q.$$
\end{lemma} 
Indeed, an integration by parts shows that $I_p^q = \frac{p}{q+1}I_{p+1}^{q+1}$. On the other hand, we can easily see that $I_p^q = I_{p+1}^q + I_{p+1}^{q+1}$. 
Together, the above relations yield the lemma.  \\

\medskip\noindent We apply Lemma \ref{art-H-S-lem-3} to the computation of $C_2(n,s)/C_1(n,s)$ when $n\geq 5$. We have that

\begin{equation}\label{art-H-S-fct-test-p1-eqt-A}
\frac{C_2(n,s)}{C_1(n,s)} = 
\frac{2}{\crit 6 n}\frac{\int_{\R^n}|X|^{2-s}|\Phi|^{\crit}dX}{\int_{\R^n}\frac{|\Phi|^{\crit}}{|X|^s}dX}\cdot \frac{\int_{\R^n}|\nabla\Phi|^2dX}{\int_{\R^n}|\Phi|^2dX}-\frac{1}{6n}\frac{\int_{\R^n}|X|^2|\nabla\Phi|^2dX}{\int_{\R^n}\Phi^2dX}
\end{equation}
Independently
$$\frac{\int_{\R^n}|X|^2|\nabla\Phi|^2dX}{\int_{\R^n}|\Phi|^2dX} = 
\frac{(n-2)^2\int_0^{+\infty}\frac{r^{n+3-2s}dr}{(1+r^{2-s})^{\frac{2(n-s)}{2-s}}}}{\int_0^{+\infty}\frac{r^{n-1}dr}{(1+r^{2-s})^{\frac{2(n-2)}{2-s}}}}, $$
up to taking $t = r^{2-s}$ and using the Lemma \ref{art-H-S-lem-3}, we get that : 
\begin{equation}\label{art-H-S-fct-test-p1-eqt-B}
\frac{\int_{\R^n}|X|^2|\nabla\Phi|^2dX}{\int_{\R^n}|\Phi|^2dX} = \frac{\frac{(n-2)^2}{2-s}\int_0^{+\infty}\frac{t^{\frac{n}{2-s}+1}dt}{(1+t)^{\frac{2(n-2)}{2-s}}}}{\frac{1}{2-s}\int_0^{+\infty}\frac{t^{\frac{n}{2-s}+2}dt}{(1+t)^{\frac{2(n-2)}{2-s}}}}
= \frac{n(n-2)(n+2-s)}{2(2n-2-s)},
\end{equation} 
\begin{equation}\label{art-H-S-fct-test-p1-eqt-C}
\frac{\int_{\R^n}|X|^{2-s}\cdot|\Phi|^{\crit}dX}{\int_{\R^n}|\Phi|^2dX} = \frac{n(n-4)}{2(n-2)(2n-2-s)}
\end{equation} 
and 
\begin{equation}\label{art-H-S-fct-test-p1-eqt-D}
 \frac{\int_{\R^n}|\nabla\Phi|^2dX}{\int_{\R^n}|X|^{-s}\cdot|\Phi|^{\crit}dX} = (n-2)(n-s).
\end{equation} 

Therefore, plugging  \eqref{art-H-S-fct-test-p1-eqt-B},  \eqref{art-H-S-fct-test-p1-eqt-C} and \eqref{art-H-S-fct-test-p1-eqt-D} into \eqref{art-H-S-fct-test-p1-eqt-A} yields
$$\frac{C_2(n,s)}{C_1(n,s)}= \frac{(n-2)(6-s)}{12(2n-2-s)}$$
when $n\geq 5$. As a conclusion, the expansion \eqref{art-H-S-fct-test-p1-eqt-9} rewrites
\begin{equation}\label{art-H-S-exp-n-5}
J(u_\epsilon) =  \frac{1}{K(n,s)}\left(1+
\left\{
\begin{array}{ll}
\displaystyle{\frac{\int_{\R^n}|\Phi|^2dX}{\int_{\R^n}|\nabla\Phi|^2dX}}\left(a(x_0)-c_{n,s}\hbox{Scal}_g(x_0)\right)\epsilon^2+o(\epsilon^2) &\hbox{ if }n\geq 5\\
\\
  \displaystyle{\frac{\omega_3}{\int_{\R^4}|\nabla\Phi|^2dX}}\left(a(x_0) - c_{n,s}\hbox{Scal}_g(x_0)\right)\epsilon^2ln(\frac{1}{\epsilon}) + O(\epsilon^2)&\hbox{ if }n=4\\
  \\
O(\epsilon)& \hbox{ if }n=3
\end{array}\right\}\right)
\end{equation}
where
\begin{equation}\label{art-H-S-fct-test-p1-eqt-13}
c_{n,s}: = \frac{(n-2)(6-s)}{12(2n-2-s)}.
\end{equation}
As a consequence, we then get the following theorem:
\begin{theorem}\label{art-H-S-th-4} Let $(M,g)$ be a compact Riemannian Manifold of dimension $n \geq 3$. Let $a \in C^0(M)$ such that
$\Delta_g + a$ is coercive, $x_0 \in M$ and $s \in (0,2)$. Then for all $n \geq 3$, we have that 
\begin{equation}\label{ineq:large}
\inf_{v\in H_1^2(M)\setminus\{0\}}J(v) \leq K(n,s)^{-1}. 
\end{equation}
Moreover, if $n\geq 4$ and $a(x_0) < c_{n,s}Scal_{g}(x_0)$, where $c_{n,s}$ is as \eqref{art-H-S-fct-test-p1-eqt-13}, then inequality \eqref{ineq:large} is strict.
\end{theorem}

\section{Test-functions: the case $n=3$}
The argument used for $n\geq 4$ is local in the sense that the expansion \eqref{art-H-S-exp-n-5} only involves the values of $a$ and $Scal_g$ at the singular point $x_0$. When $n=3$, the first-order 
in \eqref{art-H-S-exp-n-5} of Section 3 has an undetermined sign. It is well-known since Schoen \cite{Schoen-1} that the relevant quantity to use in small dimension is the mass, which is a global quantity.

\medskip\noindent We follow the technique developed by Druet \cite{Druet-1} for test-function in dimension $3$. The case of a manifold with boundary is discussed at the end of this section. We define the Green-function $G_{x_0}$ of the elliptic operator $\Delta_g + a$ on $x_0$ as the unique function strictly positive and symmetric verifying, in the sens of distribution, 
\begin{equation}\label{art-H-S-fct-test-p2-eqt-1}
 \Delta_gG_{x_0} + aG_{x_0} = D_{x_0},
\end{equation}
where $D_{x_0}$ is the Dirac mass at $x_0$. We fix $\rho \in (0, i_g(M)/2)$ and we consider a cut-off function $\eta \in C_c^\infty(\B_{2\rho}(x_0))$ such that $\eta \equiv 1$ on $\B_{\rho}(x_0)$. Then 
there exists $ \beta_{x_0} \in H_1^2(M)$ such that we can write $G_{x_0}$ as follow : 
\begin{equation}\label{art-H-S-fct-test-p2-eqt-2}
\omega_2G_{x_0} (x)= \frac{\eta(x)}{d_g(x,x_0)} + \beta_{x_0}(x)
\end{equation}
for all $x\in M$.  According to \eqref{art-H-S-fct-test-p2-eqt-1} and \eqref{art-H-S-fct-test-p2-eqt-2}, we have that
\begin{equation}\label{art-H-S-fct-test-p2-eqt-3}
 \Delta_g\beta_{x_0} + a\beta_{x_0} = f_{x_0}
\end{equation} 
where
\begin{equation}\label{art-H-S-fct-test-p2-eqt-4}
f_{x_0}(x):= -\Delta_g\left(\frac{\eta(x)}{d_g(x,x_0)}\right) - \frac{(a\eta)(x)}{d_g(x,x_0)}    \hbox{ for all }x\in M\setminus\{x_0\}.
\end{equation}
In particular, for all $p \in (1, 3) $, we have $f_{x_0}\in L^p(M)$. Therefore, it follows from standard elliptic theory that 
 $\beta_{x_0} \in C^0(M) \cap C^1_{loc}(M\setminus\{x_0\})\cap H_2^p(M)$ for all $p\in (1,3)$. In particular, the mass satisfies $m(x_0)=\beta_{x_0}(x_0)$. For any $\epsilon>0$, we define, on $M$, the function
$$v_\epsilon = \eta u_\epsilon + \sqrt{\epsilon}\beta_{x_0},$$
where $u_\epsilon$ is the general test-function defined as \eqref{art-H-S-fct-test-p1-eqt-1}.
\medskip\noindent This section is devoted to computing the expansion of $J(v_\epsilon)$. We compute the different terms separately.

\subsection{The leading term $\int_M(|\nabla v_\epsilon|_g^2 + av_\epsilon^2)dv_g$}
Integration by parts and using the definition of $v_\epsilon$, we write, for any $\epsilon > 0$, that : 
\begin{eqnarray}\label{art-H-S-fct-test-p2-eqt-5}
&&\int_M(|\nabla v_\epsilon|^2_g + av_\epsilon^2)dv_g = \int_M\eta^2u_\epsilon\Delta_gu_\epsilon dv_g + \int_Mu_\epsilon^2\eta\Delta_g\eta dv_g 
\\ &&  - \int_M\eta(\nabla\eta, \nabla u_\epsilon^2)_gdv_g + \int_Ma\eta^2u_\epsilon^2dv_g + \int_M(\Delta_g\beta_{x_0} + a\beta_{x_0})(\epsilon\beta + 2\sqrt{\epsilon}\eta u_\epsilon)dv_g.\nonumber
\end{eqnarray}
Writing $u_\epsilon^2$ in the form :
\begin{equation}\label{art-H-S-fct-test-p2-eqt-5.1}
u_\epsilon^2(x) = \frac{\epsilon}{d_g(x,x_0)^2} + O(\epsilon^{5-2s}),
\end{equation} 
with $O(1) \in C^2(M\setminus\B_\rho(x_0))$ uniformly bounded with respect to $\epsilon$,
we obtain that 
\begin{equation}\label{art-H-S-fct-test-p2-eqt-6.1}
 \int_Mu_\epsilon^2\eta\Delta_g\eta dv_g = \epsilon\int_{M\setminus\B_\rho(x_0)}\frac{\eta\Delta_g\eta}{d_g(x,x_0)^2}dv_g + o(\epsilon),
\end{equation}
and 
\begin{equation}\label{art-H-S-fct-test-p2-eqt-6.2}
 \int_M\eta(\nabla\eta, \nabla u_\epsilon^2)_gdv_g = \epsilon\int_{M\setminus\B_\rho(x_0)}\eta(\nabla\eta, \nabla\frac{1}{d_g(x,x_0)^2})_gdv_g + o(\epsilon).
\end{equation}
By integrating by parts, using \eqref{art-H-S-fct-test-p2-eqt-5.1} and since $\partial_\nu\eta = 0$ then we write 
\begin{equation}\label{art-H-S-fct-test-p2-eqt-6}
 \int_Mu_\epsilon^2\eta\Delta_g\eta dv_g -  \int_M\eta(\nabla\eta, \nabla u_\epsilon^2)_gdv_g =  \epsilon\int_{M\setminus\B_\rho(x_0)}\frac{|\nabla\eta|_g^2}{d_g(x,x_0)^2}dv_g + o(\epsilon)
\end{equation}
We have also that 
\begin{equation}\label{art-H-S-fct-test-p2-eqt-7.0}
 \int_Ma\eta^2u_\epsilon^2dv_g = \epsilon\int_M\frac{a\eta^2}{d_g(x,x_0)^2}dv_g + R_1(\epsilon) + o(\epsilon),
\end{equation}
where, as in \eqref{art-H-S-fct-test-p1-eqt-3},
$$R_1(\epsilon) = O\left(\epsilon^{3-s}\int_{\B_\rho(x_0)}\frac{a\eta^2dv_g}{d_g(x,x_0)^2(\epsilon^{2-s} + d_g(x,x_0)^{2-s})}\right)=O\left(\epsilon^2\int_0^{\frac{\rho}{\epsilon}}\frac{dt}{1 + t^{2-s}}\right) = o(\epsilon). $$
This latest relation and \eqref{art-H-S-fct-test-p2-eqt-7.0} give that 
\begin{equation}\label{art-H-S-fct-test-p2-eqt-7}
 \int_Ma\eta^2u_\epsilon^2dv_g = \epsilon\int_M\frac{a\eta^2}{d_g(x,x_0)^2}dv_g + o(\epsilon).
\end{equation}
Writing now $u_\epsilon$ in the form
\begin{equation}\label{art-H-S-fct-test-p2-eqt-8}
u_\epsilon(x) =  \frac{\sqrt{\epsilon}}{d_g(x,x_0)} + O(\epsilon^{\frac{5}{2}-s}),
\end{equation}
with $O(1) \in C^2(M\setminus\B_\rho(x_0))$
we get  that 
\begin{equation}\label{art-H-S-fct-test-p2-eqt-9}
\int_{M\setminus\B_\rho(x_0)}\eta^2u_\epsilon\Delta_g u_\epsilon dv_g = 
\epsilon\int_{M\setminus\B_\rho(x_0)}\frac{\eta^2}{d_g(x,x_0)}\Delta_g(\frac{1}{d_g(x,x_0)})dv_g + o(\epsilon).
\end{equation}
Since $u_\epsilon$ is radially symmetrical, denoting $\Delta_\delta$ as the Laplacian in the Euclidean metric $\delta$, we get with a change of variable and Cartan's expansion of the metric \eqref{art-H-S-fct-test-p1-eqt-2} that
$$\int_{\B_\rho(x_0)}u_\epsilon\Delta_\delta u_\epsilon dv_g = \int_{\R^n}\Phi\Delta_\delta\Phi dX + o(\epsilon),$$
where  $\Phi$ is defined in \eqref{art-H-S-fct-test-p1-eqt-1.1}. Since $$\Delta_g u_\epsilon = \Delta_\delta u_\epsilon - \partial_r(\ln det(g))\partial_r u_\epsilon$$
in $g-$normal coordinates, we have that
\begin{eqnarray}\label{art-H-S-fct-test-p2-eqt-10}
\int_{\B_\rho(x_0)}u_\epsilon\Delta_g u_\epsilon dv_g &=& \int_{\R^n}\Phi\Delta_\delta\Phi dX + o(\epsilon)\nonumber
\\ && + \epsilon\int_{\B_\rho(x_0)}\frac{d_g(x,x_0)^{1-s}\partial_r(\ln det(g))}{2(\epsilon^{2-s} + d_g(x,x_0)^{2-s})^{\frac{4-s}{2-s}}}dv_g+o(\epsilon)
\end{eqnarray}
when $\epsilon\to 0$. Similar computations to the ones we just developed give that 
\begin{eqnarray*}
&&\epsilon\int_{\B_\rho(x_0)}\frac{d_g(x,x_0)^{1-s}\partial_r(\ln det(g))}{2(\epsilon^{2-s} + d_g(x,x_0)^{2-s})^{\frac{4-s}{2-s}}}dv_g =
\epsilon\int_{\B_\rho(x_0)}\frac{\partial_r(\ln det(g))}{2d_g(x,x_0)^3}dv_g\\ 
 && +O\left(\epsilon^{3-s}\int_{\B_\rho(x_0)}\frac{\partial_r(\ln det(g))}{d_g(x,x_0)^3(\epsilon^{2-s} + d_g(x,x_0)^{2-s})}dv_g\right),
\\ &&
\end{eqnarray*}
Cartan's expansion of the metric $g$, \eqref{art-H-S-fct-test-p1-eqt-2} and to this latest relation yield 
\begin{eqnarray}\label{art-H-S-fct-test-p2-eqt-12}
\epsilon\int_{\B_\rho(x_0)}\frac{d_g(x,x_0)^{1-s}\partial_r(\ln det(g))}{2(\epsilon^{2-s} + d_g(x,x_0)^{2-s})^{\frac{4-s}{2-s}}}dv_g &=& 
\epsilon\int_{\B_\rho(x_0)}\frac{\partial_r(\ln det(g))}{2d_g(x,x_0)^3}dv_g + o(\epsilon) \nonumber
\\ && 
\end{eqnarray}
Relations \eqref{art-H-S-fct-test-p2-eqt-9}, \eqref{art-H-S-fct-test-p2-eqt-10} and \eqref{art-H-S-fct-test-p2-eqt-12} yield 
\begin{eqnarray}\label{art-H-S-fct-test-p2-eqt-13}
\int_M\eta^2u_\epsilon\Delta_g u_\epsilon dv_g &=& \int_{\R^n}\Phi\Delta_\delta\Phi dX +  
\epsilon\int_{\B_\rho(x_0)}\frac{\partial_r(\ln det(g))}{2d_g(x,x_0)^3}dv_g \nonumber
\\ && + \epsilon\int_{M\setminus\B_\rho(x_0)}\frac{\eta^2}{d_g(x,x_0)}\Delta_g(\frac{1}{d_g(x,x_0)})dv_g+ o(\epsilon) 
\end{eqnarray}
when $\epsilon\to 0$. At last, using again the expansion \eqref{art-H-S-fct-test-p2-eqt-13} of $u_\epsilon$, we obtain that :
\begin{eqnarray*}
\int_M(\Delta_g\beta_{x_0} + a\beta_{x_0})(\epsilon\beta_{x_0} + 2\sqrt{\epsilon}\eta u_\epsilon)dv_g &=& 
 \epsilon\int_M(\Delta_g\beta_{x_0} + a\beta_{x_0})(\beta_{x_0} + 2\frac{\eta}{d_g(x,x_0)})dv_g + 
\\ && O\left(\epsilon^{3-s}\int_M\frac{(\Delta_g\beta_{x_0} + a\beta_{x_0})\eta}{d_g(x,x_0)(\epsilon^{2-s} + d_g(x,x_0)^{2-s})}dv_g\right).
\end{eqnarray*} 
The latest relation and \eqref{art-H-S-fct-test-p2-eqt-3} allow to write : 
\begin{eqnarray*}
\int_M(\Delta_g\beta_{x_0} + a\beta_{x_0})(\epsilon\beta_{x_0} + 2\sqrt{\epsilon}\eta u_\epsilon)dv_g &=& 
 \epsilon\int_M(\Delta_g\beta_{x_0} + a\beta_{x_0})(\beta_{x_0} + 2\frac{\eta}{d_g(x,x_0)})dv_g + o(\epsilon).
\\ && 
\end{eqnarray*}
Since $\beta_{x_0}\in C^0(M)\cap H_2^p(M)$ for all $p \in (\frac{3}{2}, 3)$, it follows from \eqref{art-H-S-fct-test-p2-eqt-1} and \eqref{art-H-S-fct-test-p2-eqt-2} that  
$$\int_M(\Delta_g\beta_{x_0} + a\beta_{x_0})(\beta_{x_0} + \frac{\eta}{d_g(x,x_0)})dv_g =  \omega_2\beta_{x_0}(x_0).$$
Then the last couple of relations give that 
\begin{eqnarray}\label{art-H-S-fct-test-p2-eqt-14}
\int_M(\Delta_g\beta_{x_0} + a\beta_{x_0})(\epsilon\beta_{x_0} + 2\sqrt{\epsilon}\eta u_\epsilon)dv_g &=& \epsilon\omega_2\beta_{x_0}(x_0) + o(\epsilon)\nonumber
\\ && \ \  + \epsilon\int_M(\Delta_g\beta_{x_0} + a\beta_{x_0})(\beta_{x_0} + \frac{\eta}{d_g(x,x_0)})dv_g
\end{eqnarray}
when $\epsilon\to 0$. Knowing, from \eqref{art-H-S-fct-test-p2-eqt-3} and \eqref{art-H-S-fct-test-p2-eqt-4}, that 
\begin{eqnarray*}
\int_M(\Delta_g\beta_{x_0} + a\beta_{x_0})\frac{\eta}{d_g(x,x_0)}dv_g &=& -\int_{\B_\rho(x_0)}\frac{\partial_r(\ln det(g))}{2d_g(x,x_0)^3}dv_g
 - \int_M\frac{a\eta^2}{d_g(x,x_0)^2}dv_g +\nonumber
\\ && - \int_{M\setminus\B_\rho(x_0)}\frac{\eta}{d_g(x,x_0)}\Delta_g(\frac{\eta}{d_g(x,x_0)})dv_g,
\end{eqnarray*}
 and combining \eqref{art-H-S-fct-test-p2-eqt-6}, \eqref{art-H-S-fct-test-p2-eqt-7},\eqref{art-H-S-fct-test-p2-eqt-13} and \eqref{art-H-S-fct-test-p2-eqt-14} with \eqref{art-H-S-fct-test-p2-eqt-5}, we get that  
\begin{equation}\label{art-H-S-fct-test-p2-eqt-15}
 \int_M(|\nabla v_\epsilon|^2_g + av_\epsilon^2)dv_g = \int_{\R^n}\Phi\Delta_\delta\Phi dX + \epsilon\omega_2\beta_{x_0}(x_0) + o(\epsilon)
\end{equation}
when $\epsilon\to 0$.

\subsection{Estimate of $\int_M\frac{v_\epsilon^{\crit}}{d_g(x,x_0)^s}dv_g$.}
Since $s \in (0,2)$ then $6-2s > 2$. Therefore there exists $C(s) > 0$ such that for all $X, Y \in \R$, we have : 
$$\left||X+Y|^{6-2s} - |X|^{6-2s} - (6-2s)XY|X|^{4-2s}\right| \leq C(s)\left(|X|^{4-2s}Y^2 + |X|^{6-2s}\right) $$
This  allows to write 
\begin{eqnarray}\label{art-H-S-fct-test-p2-eqt-16}
\int_M\frac{v_\epsilon^{\crit}}{d_g(x,x_0)^s}dv_g &=& \int_M\frac{(\eta u_\epsilon + \sqrt{\epsilon}\beta_{x_0})^{6-2s}}{d_g(x,x_0)^s}dv_g\nonumber
\\ &=&  \int_{\B_\rho(x_0)}\frac{(u_\epsilon + \sqrt{\epsilon}\beta_{x_0})^{6-2s}}{d_g(x,x_0)^s}dv_g + O(\epsilon^{3-s})\nonumber
\\ &=& \int_{\B_\rho(x_0)}\frac{u_\epsilon^{6-2s} + (6-2s)|u_\epsilon|^{6-2s-2}u_\epsilon\sqrt{\epsilon}\beta_{x_0}}{d_g(x,x_0)^s}dv_g + R_2(\epsilon) + o(\epsilon) \nonumber
\\ &&
\end{eqnarray}
where  
\begin{eqnarray}\label{art-H-S-fct-test-p2-eqt-17}
R_2(\epsilon) &=& O\left(\int_{\B_\rho(x_0)}\frac{u_\epsilon^{4-2s}\epsilon\beta_{x_0}^2 + \epsilon^{3-s}\beta_{x_0}^{6-2s}}{d_g(x,x_0)^s}dv_g\right)=o(\epsilon)
\end{eqnarray} 
Direct computations yield
\begin{eqnarray}\label{art-H-S-fct-test-p2-eqt-18}
\int_{\B_\rho(x_0)}\frac{u_\epsilon^{6-2s}}{d_g(x,x_0)^s}dv_g 
= \int_{\R^n}\frac{\Phi^{\crit}(X)}{|X|^s}dX + o(\epsilon),
\end{eqnarray}
when $\epsilon\to 0$. Using that $\beta_{x_0} \in C^{0,\theta}(M)$ for all $\theta \in (0,1)$, we get that
\begin{eqnarray}\label{art-H-S-fct-test-p2-eqt-19}
\int_{\B_\rho(x_0)}\frac{(6-2s)|u_\epsilon|^{6-2s-2}u_\epsilon\sqrt{\epsilon}\beta_{x_0}}{d_g(x,x_0)^s}dv_g &=& 
\epsilon^{3-s}(6-2s)\beta_{x_0}(x_0)\omega_2\int_0^\rho\frac{r^{2-s}dr}{(\epsilon^{2-s} + r^{2-s})^{\frac{5-2s}{2-s}}} + o(\epsilon)\nonumber
\\ &=& \epsilon(6-2s)\beta_{x_0}(x_0)\omega_2\int_0^{\frac{\rho}{\epsilon}}\frac{t^{2-s}dt}{(1 + t^{2-s})^{\frac{5-2s}{2-s}}} + o(\epsilon)\nonumber
\\ &=& \epsilon(6-2s)\beta_{x_0}(x_0)\omega_2\int_0^{+\infty}\frac{t^{2-s}dt}{(1 + t^{2-s})^{\frac{5-2s}{2-s}}} + o(\epsilon),
\end{eqnarray}
when $\epsilon\to 0$. Since $\Delta_\delta\Phi = (3-s)\frac{\Phi^{\crit-1}}{|X|^s}$ in $\R^n$, a changes of variable and an integration by parts yields  
\begin{eqnarray}\label{art-H-S-fct-test-p2-eqt-19.0}
\omega_2\int_0^{+\infty}\frac{t^{2-s}dt}{(1 + t^{2-s})^{\frac{5-2s}{2-s}}} &=& \int_{\R^n}\frac{\Phi^{\crit-1}(X)}{|X|^s}dX
\\ &=& (3-s)^{-1}\lim_{R\rightarrow+\infty}\int_{\partial B_R(0)}-\partial_{\nu}\Phi dX,\nonumber
\end{eqnarray}
 where $\nu$ is the normal vector field on the Euclidean ball $B_R(0)$. Since $\partial_{\nu}\Phi = -|X|^{1-s}(1 + |X|^{2-s})^{\frac{-3+s}{2-s}}$ for all $X \in \R^n$, passing to the limit in \eqref{art-H-S-fct-test-p2-eqt-19.0} yields
$$\omega_2\int_0^{+\infty}\frac{t^{2-s}dt}{(1 + t^{2-s})^{\frac{5-2s}{2-s}}} = (3-s)^{-1}\omega_2.$$
Hence, the latest relation and \eqref{art-H-S-fct-test-p2-eqt-19} give that  
\begin{equation}\label{art-H-S-fct-test-p2-eqt-20}
\int_{\B_\rho(x_0)}\frac{(6-2s)|u_\epsilon|^{6-2s-2}u_\epsilon\sqrt{\epsilon}\beta_{x_0}}{d_g(x,x_0)^s}dv_g = \epsilon\left(2\beta_{x_0}(x_0)\omega_2\right) + o(\epsilon),
\end{equation}
when $\epsilon\to 0$. Combining \eqref{art-H-S-fct-test-p2-eqt-17}, \eqref{art-H-S-fct-test-p2-eqt-18} and \eqref{art-H-S-fct-test-p2-eqt-20}  with \eqref{art-H-S-fct-test-p2-eqt-16}, we get that  
\begin{equation}\label{art-H-S-fct-test-p2-eqt-21}
\int_M\frac{v_\epsilon^{\crit}}{d_g(x,x_0)^s}dv_g = \int_{\R^n}\frac{\Phi^{\crit}(X)}{|X|^s}dX + \epsilon\left(2\beta_{x_0}(x_0)\omega_2\right) + o(\epsilon),
\end{equation}
when $\epsilon\to 0$.
\subsection{Expansion of $J(v_\epsilon)$ and proof of Theorem \ref{art-H-S-th-1}}
Equality \eqref{art-H-S-fct-test-p2-eqt-21}, 
\eqref{art-H-S-fct-test-p2-eqt-15} and \eqref{art-H-S-fct-test-p2-eqt-19.0} yield
\begin{eqnarray}\label{art-H-S-fct-test-p2-eqt-24}
J(v_\epsilon) &=& \frac{\int_M(|\nabla v_\epsilon|^2_g + av_\epsilon^2)dv_g}{\left(\int_M\frac{v_\epsilon^{\crit}}{d_g(x,x_0)^s}dv_g\right)^{\frac{2}{\crit}}}\nonumber 
\\ &=& K(3,s)^{-1}\left(1 - \epsilon\frac{2\beta_{x_0}(x_0)\omega_2}{\int_{\R^n}|x|^{-s}\Phi^{\crit}\, dx} + o(\epsilon)\right)
\end{eqnarray}
when $\epsilon\to 0$. Noting that $m(x_0)=\beta_{x_0}(x_0)$, we then get the following as a consequence of \eqref{art-H-S-fct-test-p2-eqt-24}:

\begin{theorem}\label{art-H-S-th-5} Let $(M,g)$ be a compact Riemannian Manifold of dimension $n = 3$. Let $a \in C^0(M)$ such that $\Delta_g + a$ is coercive, $x_0 \in M$ and $s \in (0,2)$. 
Assume that that the mass at $x_0$ is positive, that is $\beta_{x_0}(x_0)>0$. Then we have that 
$$\inf_{v\in H_1^2(M)\setminus\{0\}}J(v) < K(n,s)^{-1}.  $$
\end{theorem}

\medskip\noindent 
{\bfseries Proof of Theorem  \ref{art-H-S-th-1}.}
Theorem \ref{art-H-S-th-1} follows from the existence result (Theorem \ref{art-H-S-th-3}) and the upper-bounds  (Theorem \ref{art-H-S-th-4}  and Theorem \ref{art-H-S-th-5}). 

\smallskip\noindent{\bfseries Proof of Theorem \ref{art-H-S-th-Dim3}.}
As one checks, the estimates \eqref{art-H-S-exp-n-5} and \eqref{art-H-S-fct-test-p2-eqt-24} hold when $M$ is a smooth compact manifold with boundary provided $x_0$ lies in the interior. Then Theorem \ref{art-H-S-th-1} extends to such a case, and Theorem \ref{art-H-S-th-Dim3} is a corollary.

\subsection{Examples with positive mass}
\begin{proposition}\label{art-H-S-cor-ex-1} Let $(M,g)$ be a compact Riemannian Manifold of dimension $n = 3$. Let $a \in C^0(M)$ such that $\Delta_g + a$ is coercive, $x_0 \in M$ and $s \in (0,2)$. 
If $\{a \lvertneqq c_{3,0}Scal_g\}$ or $\{a \equiv c_{3,0}Scal_g$ and $(M,g)$ is not conformally equivalent to the canonical $n$-sphere$\}$ then we have that : $$\inf_{v\in H_1^2(M)\setminus\{0\}}J(v) < K(3,s)^{-1}.$$
\end{proposition}

\medskip\noindent Indeed, the positivity of the mass in this case was proved by Druet \cite{Druet-2}. We incorporate the proof for the sake of self-completeness.

\begin{lemma}\label{a-and-Green-fct-lemma} Let $(M,g)$ be a compact Riemannian Manifold of dimension $n = 3$. We consider $a, a' \in C^0(M)$ such that operators $ \Delta_g + a$ and $\Delta_g + a'$ are coercive. We denote as $G_x, G_x'$ their respective Green's function at any point $x\in M$. We assume that $a \lvertneqq a'$. Then $\beta_x > \beta'_x$ for all $x \in M$, where $\beta_x, \beta'_x \in C^{0,\theta}(M)$, $\theta \in (0,1)$ are such that
\begin{equation}\label{a-and-Green-fct-lemma-eqt-1}
\omega_2G_x = \frac{\eta_x}{d_g(x,\cdot)} + \beta_x \ and \  \omega_2G'_x = \frac{\eta_x}{d_g(x,\cdot)} + \beta'_x.
\end{equation} 
\end{lemma} 

\begin{proof} We fix $x \in M$ and we define $h_x = \beta'_x - \beta_x$, where $\beta'_x$ and $\beta_x$ are as in \eqref{a-and-Green-fct-lemma-eqt-1}. Noting $L:=\Delta_g+a$ and $L':=\Delta_g+a'$, we have that $L'(h_x) = -(a' - a)G_x \leq 0$.
Since $h_x \in H_2^p(M)$ for all $p\in (1,3)$, then for all $y \in M$, Green's formula yields 
\begin{equation*}\label{a-and-Green-fct-lemma-eqt-3}
h_x(y) = -\int_MG'_y(z)(a' - a)(z)G_x(z)\,dv_g(z).
\end{equation*} 
Therefore $h_x\leq 0$ since $a\leq a'$. Moreover, since $a\not\equiv a'$, we have that $h_x < 0$. This ends the proof.
\end{proof}


\medskip\noindent  {\bf Proof of Proposition \ref{art-H-S-cor-ex-1}} : We consider the operator $L^0 := \Delta_g +  c_{3,0}Scal_g$, $\beta^0$ the mass of $(M,g)$ corresponding to $L^0$. 
The Positive Mass Theorem (see \cite{Schoen-2}, \cite{Schoen-3}) gives that $\beta^0_x(x) \geq 0$, the equality being achieved only in the conformal class of the canonical sphere. It then follows from Lemma \ref{a-and-Green-fct-lemma} that $\beta_{x_0}(x_0) > 0$ when $\{a \lvertneqq c_{3,0}Scal_g\}$ or $\{a \equiv  c_{3,0}Scal_g$ and $(M,g)$ is not conformally equivalent to the unit $n$-sphere$\}$. It then follows from Theorem \ref{art-H-S-th-5} that $\inf_{v\in H_1^2(M)\setminus\{0\}}J(v) < K(3,s)^{-1}$.



\begin{bibdiv}
\begin{biblist}

\bib{Aubin-2}{article}{
   author={Aubin, T.},
   title={\'Equations diff\'erentielles non lin\'eaires et probl\`eme de
   Yamabe concernant la courbure scalaire},
   journal={J. Math. Pures Appl. },
   volume={55},
   date={1976},
   pages={269--296},
}
\bib{Aubin-1}{article}{
   author={Aubin, T.},
   title={Probl\`emes isop\'erim\'etriques et espaces de Sobolev},
   journal={J. Math. Pures Appl.},
   volume={11},
   date={1976},
   pages={573--598},
}
\bib{Druet-1}{article}{
   author={Druet, O.},
   title={Optimal Sobolev inequality and Extremal functions. The three-dimensional case},
   journal={Indiana University Mathematics Journal},
   volume={51},
   date={2002},
   pages={69--88}
   number={1},
}
\bib{Druet-2}{article}{
   author={Druet, O.},
   title={Compactness for Yamabe Metrics in Low Dimensions},
   journal={IMRN International Mathematics Research Notices},
   date={2004},
   pages={1143--1191}
   number={23},
}
\bib{GH-Kang}{article}{
   author={Ghoussoub, N.},
   author={Kang, X.S.}
   title={Hardy-Sobolev critical elliptic equations with boundary singularities},
   journal={AIHP-Analyse non lin\'eaire},
   volume={21},
   date={2004},
   pages={767--793}
}
\bib{GH-R1}{article}{
   author={Ghoussoub, N.},
   author={Robert, F.}
   title={The effect of curvature on the best constant in the Hardy-Sobolev inequalities},
   journal={GAFA, Geom. funct. anal.},
   volume={16},
   date={2006},
   pages={1201--1245}
}
\bib{GH-Yuan}{article}{
   author={Ghoussoub, N.},
   author={Yuan, C.},
   title={Multiple solutions for quasi-linear PDEs involving the critical
   Sobolev and Hardy exponents},
   journal={Trans. Amer. Math. Soc.},
   volume={352},
   date={2000},
   number={12},
   pages={5703--5743},
}
\bib{Gil-Tru}{book}{
   author={Gilbarg, G.},
   author={Trudinger, N.S.},
   title={Elliptic Partial Differential Equations of Second Order, Second edition},
   publisher={Grundlehren der mathematischen Wissenschaften, Springer, Berlin},
   volume={224},
   date={1983},
}
\bib{Hebey-1}{book}{
    title={Introduction \`a l'analyse non lin\'eaire sur les Vari\'et\'es},
    author={Hebey, E.},
    publisher={Diderot},
   date={1997},
   adress={Paris}
}
\bib{Hebey-2}{book}{
   title={Non linear analysis on Manifolds : Sobolev spaces and inequalities},
   author={Hebey, E.},
   date={2001},
   publisher={American Mathematical Society, Collection : Courant lecture notes in mathematics},
   adress={5 New york University, Courant institut of Mathematics sciences, New york}
}
\bib{jaber:best:constant}{article}{
   author={Jaber, H.},
   title={Sharp Constant in the Riemannian Hardy-Sobolev inequality},
note={In preparation}

}
\bib{Kang-Peng}{article}{
   author={Kang, D.},
   author={Peng, S.},
   title={Existence of solutions for elliptic equations with critical Sobolev-Hardy exponents},
   journal={Nonlinear Analysis},
   volume={56},
   date={2004},
   pages={1151--1164},
}
\bib{Lee-parker}{article}{
   author={Lee, J.},
   author={Parker, T.}
   title={The Yamabe problem},
   journal={Bull. Amer. Math. Soc. (N.S.)},
   volume={17},
   date={1987},
   pages={37--91}
   number={1}
}
\bib{Lieb-1}{article}{
   author={Lieb, E.H.},
   title={Sharp constants in the Hardy-Littlewood-Sobolev and related inequalities},
   journal={Ann. of Mathematics},
   volume={118},
   date={1983},
   pages={349--374}
}
\bib{Li-Ruf-Guo-Niu}{article}{
   author={Li, Y.},
   author={Ruf, B.},
   author={Guo, Q.},
   author={Niu, P.},
   title={Quasilinear elliptic problems with combined critical Sobolev-Hardy terms},
   journal={Annali di Matematica},
   volume={192},
   date={2013},
   pages={93--113},
}
\bib{Musina}{article}{
   author={Musina, R.},
   title={Existence of extremals for the Maz'ya and for the
   Caffarelli-Kohn-Nirenberg inequalities},
   journal={Nonlinear Anal.},
   volume={70},
   date={2009},
   number={8},
   pages={3002--3007},
 }
\bib{Pucci-Servadei}{article}{
   author={Pucci, P.},
   author={Servadei, R.},
   title={Existence, non existence and regularity of radial ground states for p-Laplacian equations with singular weights},
   journal={Ann. I. H. Poincar\'e},
   volume={25},
   date={2008},
   pages={505-537},
}
\bib{Rodemich}{article}{
   author={Rodemich, E.},
   title={The Sobolev inequalities with best possible constant},
   journal={ Analysis seminar at California Institute of technology},
   date={1966},
}
\bib{Schoen-1}{article}{
   author={Schoen, R.},
   title={Conformal deformation of a Riemannian metric to a constant scalar curvature},
   journal={Journal of Differential Geometry},
   volume={118},
   date={1984},
   pages={479--495}
}
\bib{Schoen-2}{article}{
   author={Schoen, R.},
   author={Yau, S.}
   title={On the proof of the positive mass conjecture in general relativity},
   journal={Comm. Math. Phys.},
   volume={65},
   date={1979},
   pages={45--76}
   number={1}
}
\bib{Schoen-3}{article}{
   author={Schoen, R.},
   author={Yau, S.}
   title={Proof of the positive action-conjecture in quantum relativity},
   journal={Phys. Rev. Lett.},
   volume={42},
   date={1979},
   pages={547--548}
   number={9}
}
\bib{Talenti}{article}{
   author={Talenti, G.},
   title={Best constant in Sobolev inequality},
   journal={Ann. di Matem. Pura ed Appl.},
   volume={110},
   date={1976},
   pages={353--372},
}
\bib{Elhadji-A.T.}{article}{
   author={Thiam, E.H.A.},
   title={Hardy and Hardy-Soboblev Inequalities on Riemannian Manifolds},
   date={2013},
   note={Preprint}
}
\bib{Tru-2}{article}{
   author={Trudinger, N.S.},
   title={Remarks concerning the conformal deformation of Riemannian structures on compact Manifolds},
   journal={Ann. Scuola Norm. Sup. Pisa},
   volume={22},
   date={1968},
   pages={265--274},
}
\bib{Yamabe}{article}{
   author={Yamabe, H.},
   title={On a Deformation of Riemannian Structures on Compact Manifolds},
   journal={Osaka Math. J.},
   volume={12},
   date={1960},
   pages={21--37}
}

\end{biblist}
\end{bibdiv}
\end{document}